\documentclass{amsart}

\usepackage{amscd,amsmath,amsxtra,stmaryrd,amsthm,amssymb,xr,mathrsfs,color}

\newtheorem{theorem}{Theorem}[section]
\newtheorem{lemma}[theorem]{Lemma}

\newtheorem{proposition}[theorem]{Proposition}
\newtheorem{corollary}[theorem]{Corollary}
\newtheorem{definition}[theorem]{Definition}
\newtheorem{remark}[theorem]{Remark}

\newcommand{\Qp}{\mathbb{Q}_p}
\newcommand{\Zp}{\mathbb{Z}_p}

\begin{document}

\title{Plus and minus logarithms and Amice transform}

\begin{abstract}
We give a new description of Pollack's plus and minus $p$-adic logarithms $\log_p^\pm$ in terms of distributions. In particular, if $\mu_\pm$ denote the pre-images of $\log_p^\pm$ under the Amice transform, we give explicit formulae for the values $\mu_\pm(a+p^n\Zp)$ for all $a\in \Zp$ and all integers $n\ge1$. Our formulae imply that the distribution $\mu_-$ agrees with a distribution studied by Koblitz in 1977.  Furthermore, we show that a similar description exists for  Loeffler's two-variable analogues of these plus and minus logarithms.
\end{abstract}

\author{C\'edric Dion}
\address{C\'edric Dion\newline
D\'epartement de Math\'ematiques et de Statistique\\
Universit\'e Laval, Pavillion Alexandre-Vachon\\
1045 Avenue de la M\'edecine\\
Qu\'ebec, QC\\
Canada G1V 0A6}
\email{cedric.dion.1@ulaval.ca}

\author{Antonio Lei}
\address{Antonio Lei\newline
D\'epartement de Math\'ematiques et de Statistique\\
Universit\'e Laval, Pavillion Alexandre-Vachon\\
1045 Avenue de la M\'edecine\\
Qu\'ebec, QC\\
Canada G1V 0A6}
\email{antonio.lei@mat.ulaval.ca}


\maketitle

\section{Introduction}
Let $p$ a fixed prime and $E$ an elliptic curve with good supersingular reduction at $p$. Amice-V\'elu \cite{amicevelu} and Visik \cite{visik} constructed two $p$-adic $L$-functions, $L_\alpha$ and $L_\beta$, corresponding to the two roots ($\alpha$ and $\beta$) of the Hecke polynomial $X^2-a_pX+p$. These $p$-adic $L$-functions can be regarded as distributions on the Galois group $\Zp^\times$. When evaluated at Dirichlet characters of $p$-power conductor, these distributions interpolate the complex $L$-values of $E$ twisted by these characters.

Unlike the good ordinary case, where the $p$-adic $L$-function can be used to formulate the Iwasawa main conjecture of $E$ at $p$, the $p$-adic $L$-functions in the supersingular case are less suited for the study of Iwasawa theory. The main reason is that they turn out to be unbounded distributions, rather than bounded measures on $\Zp^\times$. In the case where $a_p=0$, Pollack \cite{pollack03} resolved this issue by defining the plus/minus logarithms, which we denote by $\log_p^\pm$, and proved that there exist two bounded measures, $L_+$ and $L_-$, such that
\begin{equation}\label{eq:pollack}
L_\lambda=\log_p^+L_++\lambda\log_p^-L_-,
\end{equation}
for $\lambda\in\{\alpha,\beta\}$.

Kobayashi \cite{kobayashi03} formulated an Iwasawa main conjecture using the $p$-adic $L$-functions constructed by Pollack and showed that one inclusion of this conjecture holds true. More recently, Wan \cite{wan} has announced a proof for the other inclusion of the conjecture. 

In \cite{pollack03}, $\log_p^\pm$ are defined to be power series given by infinite products of cyclotomic polynomials. In this article, we show that these power series may be described in terms of distributions. We define the following subsets of $\Zp$:
\begin{align*}
S_n^+&:=\{ a\in\Zp : \text{all even powers of $p$ vanish in the $p$-adic expansion of $a$ modulo $p^n$}\};\\
S_n^-&:=\{ a\in\Zp: \text{all odd powers of $p$ vanish in the $p$-adic expansion of $a$ modulo $p^n$}\}.
\end{align*}
Our main result describes the distributions attached to $\log_p^\pm$ in terms of these $S_n^\pm$.
\begin{theorem}[Corollary~\ref{values}]\label{thm:main}
Let $\mu_\pm$ be the distributions whose Amice transforms equal $\log_p^\pm$. Then, $\mu_+$ is characterized by
\[
\mu_+(a+p^n\Zp)=
\begin{cases}
p^{-\lfloor (n+2)/2 \rfloor}&\text{if $a\in S_n^+$,}\\
0&\text{otherwise.}
\end{cases}
\]
The distribution $\mu_-$ is characterized by
\[
\mu_-(a+p^n\Zp)=
\begin{cases}
p^{-\lfloor (n+3)/2 \rfloor}&\text{if $a\in S_n^-$,}\\
0&\text{otherwise.}
\end{cases}
\]
\end{theorem}

It is a straightforward exercise to verify that the values given in Theorem~\ref{thm:main} are additive. In fact, the values of $\mu_-(a+p^n\Zp)$ are precisely those given in \cite[Chapter II, \S3, Exercise 7]{koblitz} (after multiplying by $p$).

Since the works of Kobayashi \cite{kobayashi03} and Pollack \cite{pollack03}, many attempts to generalize the plus/minus Iwasawa theory have been made. For example, Darmon-Iovita \cite{darmoniovita}, Longo-Vigni \cite{longovigni}, Castella-Wan \cite{castellawan} have studied the behaviour of an elliptic curve with supersingular reduction at $p$ over the anticyclotomic $\Zp$-extension of an imaginary quadratic field. Iovita-Pollack \cite{iovitapollack} have done something similar for more general $\Zp$-extensions of a number field. In all these works, the plus and minus logarithms play an important role.

Another direction that has been taken was the study of two-variable plus/minus Iwasawa theory as introduced by  Loeffler \cite{Loeffler14} and Kim \cite{Kim14,kimselmer}. More specifically, they studied the $p$-adic $L$-functions and Selmer groups of $E$ over the $\Zp^2$-extension of an imaginary quadratic field $K$ where $p$ splits. This leads to the construction of four 2-variable $p$-adic $L$-functions, say $L_{\lambda,\nu}$, where $\lambda,\nu\in\{\alpha,\beta\}$. The two variables correspond to the two primes of $K$ lying above $p$. A similar decomposition to \eqref{eq:pollack} exists, namely,
\[
L_{\lambda,\nu}=\log_p^{++}L_{++}+\lambda\log_p^{-+}L_{-+}+\nu\log_p^{+-}L_{+-}+\log_p^{--}L_{--},
\]
where $\log_p^{\pm\pm}$ are some two-variable plus/minus logarithms, defined by the product of two one-variable plus/minus logarithms, and $L_{\pm\pm}$ are some bounded two-variable $p$-adic $L$-functions. We shall show that we may deduce from Theorem~\ref{thm:main} a similar description of the two-variable plus/minus logarithms in terms of some explicit distributions on $\Zp^2$ (see Corollary~\ref{cor:final2} below).

\section*{Acknowledgment}
The authors would like to thank Daniel Delbourgo for answering our questions during the preparation of this article. The authors' research was supported by the  NSERC Discovery Grants Program 05710 and an ISM summer scholarship.

\section{Definition and properties of Pollack's plus/minus logarithms}

In this section, we recall the definition and some basic properties of Pollack's plus/minus logarithms from  \cite{pollack03}. We begin by introducing some notation.

Let $\Phi_n(T) = \sum_{t=0}^{p-1}T^{p^{n-1}t}$ be the $p^n$-th cyclotomic polynomial. We shall also write $\zeta_k$ for a primitive $p^k$-th root of unity.

\begin{align*}
\log_{p}^{+}(T) &:= \frac{1}{p}\prod_{n=1}^{\infty}\frac{\Phi_{2n}(1+T)}{p}, \\ \log_{p}^{-}(T) &:= \frac{1}{p}\prod_{n=1}^{\infty}\frac{\Phi_{2n-1}(1+T)}{p}.
\end{align*}

Pollack showed that the products above converge and give two power series in $\Qp \llbracket T \rrbracket$. Furthermore, both $\log_p^\pm$ converge on the open unit disc $\{x\in \mathbb{C}_p:|x|_p<1\}$. In fact, they are related to the usual $p$-adic logarithm via the formula
\[
\log_p^+(T)\log_p^-(T)=\frac{\log_p(1+T)}{p^2T},
\]
hence their names. These logarithms satisfy the following interpolation formulae.

\begin{lemma}\label{interpolation}
Let $n\ge 1$ be an integer. Then,
\begin{align*}
\log_p^{+}(\zeta_n-1) &= \begin{cases} 0 & 2|n, \\ p^{-(n+1)/2}\prod_{j=1}^{(n-1)/2}\Phi_{2j}(\zeta_n) & 2 \nmid n, \end{cases} \\
\log_p^{-}(\zeta_n-1) &= \begin{cases} p^{-n/2-1}\prod_{j=1}^{n/2}\Phi_{2j-1}(\zeta_n) & \quad 2|n, \\ 0 & \quad 2 \nmid n. \end{cases}
\end{align*}
\end{lemma}

\begin{proof}
This is \cite[Lemma~4.7]{pollack03}.
\end{proof}

\begin{remark}\label{rk:interpolate} Let $k<n$ be two odd positive integers, then 
\begin{equation*}
p^{-(k+1)/2}\prod_{j=1}^{(k-1)/2}\Phi_{2j}(\zeta_k) = p^{-(n+1)/2}\prod_{j=1}^{(n-1)/2}\Phi_{2j}(\zeta_k).
\end{equation*}
This is because,  every term with index greater than $(k-1)/2$ on the right-hand side equals $p$. So, it does not change the value of the product as it gets cancelled by the extra factors of $p$ in $p^{-(n+1)/2}$. In other words, we may write
\[
\log_p^{+}(\zeta_k-1) = p^{-(n+1)/2}\prod_{j=1}^{(n-1)/2}\Phi_{2j}(\zeta_k) .
\]
Similarly, we have
\[
\log_p^{-}(\zeta_k-1) =  p^{-n/2-1}\prod_{j=1}^{n/2}\Phi_{2j-1}(\zeta_k) 
\]
if $k<n$ are two even positive integers.
\end{remark}

\begin{remark}
It is proved in \cite[Lemma~4.5]{pollack03} that both $\log_p^\pm$ are $o(\log_p)$. Hence, they are uniquely determined by the interpolation formulae given in Lemma~\ref{interpolation}.
\end{remark}

\section{Proof of Theorem~\ref{thm:main}}
In this section, we give a detailed proof of the main result of this article Theorem~\ref{thm:main}. We begin by some preliminary lemmas.

The original definition of the plus and minus logarithms are given by infinite products of even and odd $p$-power cyclotomic polynomials respectively. We shall express the truncated product of these polynomials in the form $\sum a_ix^i$. 

For an integer $n\ge1$, we define 
\begin{align*}
R_n^+&=\left\{ \sum_{l=0}^{n-1} a_lp^{2l+1} : a_l \in \{0,1,\ldots,p-1\} \right\},\\
R_n^-&=\left\{ \sum_{l=0}^{n-1} a_lp^{2l} : a_l \in \{0,1,\ldots,p-1\} \right\}.
\end{align*}
These sets, which are related to the sets $S_n^\pm$ given in the introduction, will play an important role in our proof of Theorem~\ref{thm:main}.

\begin{lemma}\label{cycleven}
Let $n \geq 1$. We have the formula
\begin{equation*}
\prod_{j=1}^{n}\Phi_{2j}(x) = \sum_{i\in R_n^+} x^i.
\end{equation*}
\end{lemma}

\begin{proof}
We prove this formula by induction on $n$. For the base case $n=1$, the sum on the right-hand side of the formula is  simply the definition of the cyclotomic polynomial $\Phi_2(x)$.

We now suppose that the lemma holds for $n$ and we shall show that it also holds for $n+1$. 
 Our inductive hypothesis implies that
\begin{equation*}
\begin{split}
\prod_{j=1}^{n+1}\Phi_{2j}(x) &= \Phi_{2(n+1)}(x)\prod_{j=1}^{n}\Phi_{2j}(x) \\
&=\left(x^{p^{2n+1}(p-1)} + \cdots + x^{p^{2n+1}} +1\right)\sum_{i\in R_n^+} x^i\\
&=\sum_{j=0}^{p-1}\sum_{i\in R_n^+} x^{i+jp^{2n+1}} .
\end{split}
\end{equation*}
But it is clear from definition that
\[
R_{n+1}^+=\left\{i+jp^{2n+1}:i\in R_n^+,0\le j \le p-1\right\}.
\]
This completes the induction.
\end{proof}

We have the following analogous result for the product of odd  $p$-power cyclotomic polynomials.
\begin{lemma}\label{cyclodd}
Let $n \geq 1$  be an integer. Then,
\begin{equation*}
\prod_{j=1}^{n}\Phi_{2j-1}(x) = \sum_{i\in R_n^-} x^{i}.
\end{equation*}
\end{lemma}

\begin{proof}The proof is analogous to that of Lemma~\ref{cycleven}.
\end{proof}

We now recall the definition of the Amice transform.

\begin{definition}
The Amice transform of a distribution $\mu$ on $\Zp$ is defined to be $$A_{\mu}(T) = \int_{\Zp} (1+T)^x \mu(x).$$
\end{definition}

\begin{definition}
Let $\mu_{+}$ be the distribution associated to $\log_p^{+}$ and let $\mu_{-}$ be that associated to $\log_p^{-}$.
\end{definition}

If $z\in \mathbb{C}_p$ satisfies $|z|_p<1$, it is immediate from the definition of $\mu_\pm$ that $\int_{\Zp}z^x \mu_{\pm}(x) = \log_p^{\pm}(z -1)$.

\begin{proposition}\label{inter}
Let $n\ge1$ be an integer and fix $a\in \Zp$. 
The distribution $\mu_{+}$ satisfies
\begin{equation*}
\mu_{+}(a+p^n\Zp) = \frac{1}{p^{\lfloor(3n+2)/2\rfloor}}\sum_{\zeta \in \mu_{p^n}}\zeta^{-a}\prod_{j=1}^{\lfloor n/2\rfloor}\Phi_{2j}(\zeta)
\end{equation*}
and the distribution $\mu_{-}$ satisfies
\begin{equation*}
\mu_{-}(a+p^n\Zp) = \frac{1}{p^{\lfloor (3n+1)/2\rfloor+1}}\sum_{\zeta \in \mu_{p^n}}\zeta^{-a}\prod_{j=1}^{\lfloor (n+1)/2 \rfloor}\Phi_{2j-1}(\zeta).
\end{equation*}
\end{proposition}

\begin{proof}
We shall only give the proof for $\mu_+$ since that for $\mu_-$ is similar.
Let us write $\chi_{a+p^n\Zp}$ for the characteristic function of $a+p^n\Zp$. It can be decomposed as 
\[\chi_{a+p^n\Zp}(x) = \frac{1}{p^n}\sum_{\zeta \in \mu_{p^n}} \zeta^{x-a}.\]
By linearity, we have
\begin{equation*}
\begin{split}
\mu_{+}(a+p^n\Zp) &= \int_{\Zp}\chi_{a+p^n\Zp}(x)\mu_{+}(x) \\
&= \sum_{\zeta \in \mu_{p^n}} \frac{\int_{\Zp} \zeta^{x-a}\mu_{+}(x)}{p^n} \\
&=\frac{1}{p^n}\sum_{\zeta \in \mu_{p^n}}\zeta^{-a}\int_{\Zp}\zeta^x\mu_{+} \\
&=\frac{1}{p^n}\sum_{\zeta \in \mu_{p^n}}\zeta^{-a}\log_p^{+}(\zeta-1)
\end{split}
\end{equation*}

Suppose that $n$ is odd. Lemma~\ref{interpolation} together with Remark~\ref{rk:interpolate} allow us to rewrite this sum as
\begin{equation*}
\begin{split}
\mu_{+}(a+p^n\Zp) &= \frac{1}{p^n} \sum_{\zeta \in \mu_{p^n}}\zeta^{-a}p^{-(n+1)/2} \prod_{j=1}^{(n-1)/2} \Phi_{2j}(\zeta)\\
&= \frac{1}{p^{(3n+1)/2}} \sum_{\zeta \in \mu_{p^n}}\zeta^{-a} \prod_{j=1}^{(n-1)/2} \Phi_{2j}(\zeta).
\end{split}
\end{equation*}

When $n$ is even, $n+1$ is odd and every $\zeta$ in the sum above is a $p^{n+1}$-st root of unity. On applying  Lemma~\ref{interpolation} and Remark~\ref{rk:interpolate} with $n$ replaced by  $n+1$,  we deduce that
\begin{align*}
\mu_{+}(a+p^n\Zp) &=\frac{1}{p^n} \sum_{\zeta \in \mu_{p^n}}\zeta^{-a}p^{-(n+2)/2} \prod_{j=1}^{n/2} \Phi_{2j}(\zeta)\\
&=\frac{1}{p^{(3n+2)/2}} \sum_{\zeta \in \mu_{p^n}}\zeta^{-a} \prod_{j=1}^{n/2} \Phi_{2j}(\zeta).
\end{align*}
This finishes the proof.
\end{proof}

\begin{corollary}\label{final}Let $a$ and $n$ be as given in Proposition~\ref{inter}. Then,
\begin{align*}
\mu_{+}(a+p^n\Zp) &= \frac{1}{p^{\lfloor (3n+2)/2 \rfloor}}\sum_{i\in R_{\lfloor n/2\rfloor}^+}\sum_{\zeta \in \mu_{p^n}}\zeta^{i-a}, \\
\mu_{-}(a+p^n\Zp) &=\frac{1}{p^{\lfloor (3n+1)/2 \rfloor +1}}\sum_{i\in R_{\lfloor (n+1)/2\rfloor}^{-}}\sum_{\zeta \in \mu_{p^n}}\zeta^{i-a}.
\end{align*}
\end{corollary}

\begin{proof}
This follows from combining Proposition~\ref{inter} with Lemmas~\ref{cycleven} and \ref{cyclodd}.
\end{proof}

Let us now recast the sets $S_n^\pm$ given in the introduction as follows:  
\begin{align*}
S_n^+&=\{a\in \Zp:\exists b\in R_{\lfloor n/2\rfloor}^+, a\equiv b\mod p^n\},\\
S_n^-&=\{a\in \Zp:\exists b\in R_{\lfloor (n+1)/2\rfloor}^-, a\equiv b\mod p^n\}.
\end{align*}

\begin{corollary}\label{values}
Let $a$ and $n$ be as given in Proposition~\ref{inter}. The values of $\mu_{\pm}(a+p^n\Zp)$ are given by
\begin{align*}
\mu_{+}(a+p^n\Zp) &= \begin{cases} p^{-\lfloor (n+2)/2 \rfloor} & \text{if} \ a \in S_n^{+}, \\ 0 & \text{otherwise}, \end{cases} \\
\mu_{-}(a+p^n\Zp) &=  \begin{cases} p^{-\lfloor (n+3)/2 \rfloor} & \text{if} \ a \in S_n^{-}, \\ 0 & \text{otherwise}. \end{cases}
\end{align*}
\end{corollary}

\begin{proof}
As before, we only prove this for $\mu_+(a+p^n\Zp)$. Let $i$ be any integer. We have
\[
\sum_{\zeta\in \mu_{p^n}}\zeta^{i-a}
=
\begin{cases}
p^n&{i\equiv a\mod p^n},\\
0&\text{otherwise}.
\end{cases}
\]
Therefore, if $a\in S_n^+$, the sum in Corollary~\ref{final} simplifies to
\[
\frac{1}{p^{\lfloor (3n+2)/2 \rfloor}}\times p^n.
\]
Otherwise, it is $0$. Hence the result.
\end{proof}
\begin{remark}
Since $\log_p^\pm$ are both $o(\log_p)$, the distributions $\mu_\pm$ are uniquely determined by the values given in Corollary~\ref{values}.
\end{remark}
\section{Generalization for two-variable logarithms}

In this section, we apply our result on the one-variable plus and minus logarithms to their two-variable counterparts  defined by Loeffler in \cite{Loeffler14}.

\begin{definition}
For $\ast, \circ \in \{+,-\}$, we define four two-variable logarithms by using $\log_p^{+}$ and $\log_p^{-}$:
\begin{equation*}
\log_p^{\ast\circ}(T_1,T_2) := \log_p^{\ast}(T_1) \cdot \log_p^{\circ}(T_2).
\end{equation*}
\end{definition}

For $\mathbf{a} = (a,b)$ and $\mathbf{n}=(n,m)$, we shall write $\mathbf{a}+p^{\mathbf{n}}\Zp$ for the open set $(a+p^n\Zp,b+p^m\Zp)$ in $\Zp^2$. Furthermore, we denote   the characteristic function of $\mathbf{a}+p^{\mathbf{n}}\Zp$  on $\Zp^2$ by $\chi_{\mathbf{a}+p^{\mathbf{n}}\Zp}$.
\begin{proposition}\label{characteristic2}
Let $\mathbf{a} = (a,b)$ and $\mathbf{n}=(n,m)$. The characteristic function $\chi_{\mathbf{a}+p^{\mathbf{n}}\Zp}$ is given by
\begin{equation*}
\chi_{\mathbf{a}+p^{\mathbf{n}}\Zp}(x,y) = \frac{1}{p^{n+m}}\left( \sum_{\zeta \in \mu_{p^n}} \zeta^{x-a} \right) \left(\sum_{\zeta \in \mu_{p^m}} \zeta^{y-b} \right).
\end{equation*}
\end{proposition}

\begin{proof}
This is a straightforward generalization of the analogous result in the 1-dimensional case.
\end{proof}

\begin{definition}
For $\mu \in \mathcal{D}(\Zp^2,\Qp)$,  the \textit{two-dimensional Amice transform} is defined by the formula
\begin{equation*}
A_{\mu}(T_1,T_2) = \int_{\Zp^2}(1+T_1)^x(1+T_2)^y \mu(x,y).
\end{equation*}
\end{definition}

\begin{remark}
In the notation of  \cite{Loeffler14}, the two-variable plus/minus logarithms are in fact contained $\mathcal{D}^{(1/2,1/2)}(\Zp^2,\Qp)$ under the canonical quasi-factorisation of $\Zp^2$. Consequently, they are uniquely determined by their values evaluated on open sets of the form $\mathbf{a}+p^\mathbf{n}\Zp$ (see Definition~5 and Theorem 3 of \textit{op. cit.})
\end{remark}

It is immediate from the definition that $\log_p^{\ast\circ}(z-1,w-1) = \int_{\Zp^2}z^xw^y\mu_{\ast\circ}(x,y)$ if $z$ and $w$ are elements of $\mathbb{C}_p$ with $|z|_p,|w|_p<1$.

\begin{proposition}\label{prop:2var}
Let $\mathbf{a} = (a,b)$ and $\mathbf{n}=(n,m)$, then
\begin{equation*}
\mu_{\ast\circ}(\mathbf{a}+p^{\mathbf{n}}\Zp) = \mu_{\ast}(a+p^n\Zp)\cdot\mu_{\circ}(b+p^m\Zp).
\end{equation*}
\end{proposition}

\begin{proof}
We follow the same argument as in the one-dimensional case. Proposition~\ref{characteristic2} gives
\begin{equation*}
\mu_{\ast\circ}(\mathbf{a}+p^{\mathbf{n}}\Zp) = \frac{1}{p^{n+m}}\int_{\Zp^2}\left( \sum_{\zeta \in \mu_{p^n}} \zeta^{x-a} \right) \left(\sum_{\zeta \in \mu_{p^m}} \zeta^{y-b} \right) \mu_{\ast\circ}(x,y).
\end{equation*}
Choose $\zeta_n$ and $\zeta_m$ to be some generators of $\mu_{p^n}$ and $\mu_{p^m}$ respectively. Then,
\begin{equation*}
\begin{split}
\mu_{\ast\circ}(\mathbf{a}+p^{\mathbf{n}}\Zp) &= \frac{1}{p^{n+m}}\int_{\Zp^2}\sum_{{0 \leq i < p^n}\atop{0 \leq j < p^m}}(\zeta_n^i)^{x-a}(\zeta_m^j)^{y-b}\mu_{\ast\circ}(x,y) \\
&=\frac{1}{p^{n+m}}\sum_{{0 \leq i < p^n}\atop{0 \leq j < p^m}}\zeta_n^{-ai} \zeta_m^{-bj}\int_{\Zp^2}(\zeta_n^{i})^x(\zeta_m^{j})^y\mu_{\ast\circ}(x,y) \\
&=\frac{1}{p^{n+m}}\sum_{{0 \leq i < p^n}\atop{0 \leq j < p^m}}\zeta_n^{-ai} \zeta_m^{-bj}\log_p^{\ast\circ}(\zeta_n^i-1,\zeta_m^j-1) \\
&=\frac{1}{p^{n+m}}\sum_{{0 \leq i < p^n}\atop{0 \leq j < p^m}}\zeta_n^{-ai}\log_p^{\ast}(\zeta_n^i-1)\zeta_m^{-bj}\log_p^{\circ}(\zeta_m^j-1) \\
&=\frac{1}{p^{m+n}} \left( \sum_{\zeta \in \mu_{p^n}} \zeta^{-a}\log_p^{\ast}(\zeta-1) \right) \left( \sum_{\zeta \in \mu_{p^m}} \zeta^{-b}\log_p^{\circ}(\zeta-1) \right).
\end{split}
\end{equation*}
Then, as in the middle of the proof of proposition~\ref{inter}, we may rewrite this as
\begin{equation*}
\begin{split}
\mu_{\ast\circ}(\mathbf{a}+p^{\mathbf{n}}\Zp) &= \frac{p^np^m}{p^{m+n}}\mu_{\ast}(a+p^n\Zp)\cdot\mu_{\circ}(b+p^m\Zp) \\
&=\mu_{\ast}(a+p^n\Zp)\cdot\mu_{\circ}(b+p^m\Zp),
\end{split}
\end{equation*}
as required.
\end{proof}

\begin{corollary}\label{cor:final2}
The values of $\mu_{\ast\circ}$ are given by
\begin{align*}
\mu_{++}(\mathbf{a}+p^{\mathbf{n}}\Zp) &= \begin{cases} p^{-\lfloor (n+2)/2 \rfloor -\lfloor (m+2)/2 \rfloor} & \text{if} \ a \in S_n^{+} \textit{and} \ b \in S_m^{+}, \\ 0 & \text{otherwise}, \end{cases} \\
\mu_{+-}(\mathbf{a}+p^{\mathbf{n}}\Zp) &= \begin{cases} p^{-\lfloor (n+2)/2 \rfloor -\lfloor (m+3)/2 \rfloor} & \text{if} \ a \in S_n^{+} \textit{and} \ b \in S_m^{-}, \\ 0 & \text{otherwise}, \end{cases} \\
\mu_{-+}(\mathbf{a}+p^{\mathbf{n}}\Zp) &= \begin{cases} p^{-\lfloor (n+3)/2 \rfloor -\lfloor (m+2)/2 \rfloor} & \text{if} \ a \in S_n^{-} \textit{and} \ b \in S_m^{+}, \\ 0 & \text{otherwise}, \end{cases} \\
\mu_{--}(\mathbf{a}+p^{\mathbf{n}}\Zp) &= \begin{cases} p^{\lfloor (n+3)/2 \rfloor -\lfloor (m+3)/2 \rfloor} & \quad \text{if} \ a \in S_n^{-} \textit{and} \ b \in S_m^{-}, \\ 0 &  \quad \text{otherwise}. \end{cases} \\
\end{align*}
\end{corollary}

\begin{proof}
This follows directly from combining Proposition~\ref{prop:2var} with Corollary~\ref{values}.
\end{proof}

\bibliographystyle{amsalpha}
\bibliography{references}

\providecommand{\bysame}{\leavevmode\hbox to3em{\hrulefill}\thinspace}
\providecommand{\MR}{\relax\ifhmode\unskip\space\fi MR }
\providecommand{\MRhref}[2]{%
  \href{http://www.ams.org/mathscinet-getitem?mr=#1}{#2}
}
\providecommand{\href}[2]{#2}
\begin{thebibliography}{Kim14b}

\bibitem[AV75]{amicevelu}
Yvette Amice and Jacques V\'elu, \emph{Distributions {$p$}-adiques associ\'ees
  aux s\'eries de {H}ecke}, Ast\'erisque (1975), no.~24-25, 119--131.

\bibitem[CW16]{castellawan}
Francesc Castella and Xin Wan, \emph{{$\Lambda$}-adic {G}ross-{Z}agier formula
  for supersingular primes}, arXiv:1607.02019, 2016.

\bibitem[DI08]{darmoniovita}
Henri Darmon and Adrian Iovita, \emph{The anticyclotomic main conjecture for
  elliptic curves at supersingular primes}, J. Inst. Math. Jussieu \textbf{7}
  (2008), no.~2, 291--325.

\bibitem[IP06]{iovitapollack}
Adrian Iovita and Robert Pollack, \emph{Iwasawa theory of elliptic curves at
  supersingular primes over {$\ZZ_p$}-extensions of number fields}, J. Reine
  Angew. Math. \textbf{598} (2006), 71--103.

\bibitem[Kim14a]{kimselmer}
Byoung~Du Kim, \emph{Signed-{S}elmer groups over the {$\Bbb{Z}_p^2$}-extension
  of an imaginary quadratic field}, Canad. J. Math. \textbf{66} (2014), no.~4,
  826--843.

\bibitem[Kim14b]{Kim14}
\bysame, \emph{Two-variable {$p$}-adic {$L$}-functions of modular forms for
  non-ordinary primes}, J. Number Theory \textbf{144} (2014), 188--218.

\bibitem[Kob77]{koblitz}
Neal Koblitz, \emph{{$p$}-adic numbers, {$p$}-adic analysis, and
  zeta-functions}, Springer-Verlag, New York-Heidelberg, 1977, Graduate Texts
  in Mathematics, Vol. 58.

\bibitem[Kob03]{kobayashi03}
Shin-Ichi Kobayashi, \emph{Iwasawa theory for elliptic curves at supersingular
  primes}, Inventiones Mathematicae \textbf{152} (2003), no.~1, 1--36.

\bibitem[Loe14]{Loeffler14}
David Loeffler, \emph{{$p$}-adic integration on ray class groups and
  non-ordinary {$p$}-adic {$L$}-functions}, Iwasawa theory 2012, Contrib. Math.
  Comput. Sci., vol.~7, Springer, Heidelberg, 2014, pp.~357--378.

\bibitem[LV15]{longovigni}
Matteo Longo and Stefano Vigni, \emph{Plus/minus {H}eegner points and {I}wasawa
  theory of elliptic curves at supersingular primes}, 2015, preprint,
  arXiv:1503.07812.

\bibitem[Pol03]{pollack03}
Robert Pollack, \emph{On the {$p$}-adic {$L$}-function of a modular form at a
  supersingular prime}, Duke Math. J. \textbf{118} (2003), no.~3, 523--558.

\bibitem[Vis76]{visik}
Misha~M. Visik, \emph{Nonarchimedean measures associated with {D}irichlet
  series}, Mat. Sb. (N.S.) \textbf{99(141)} (1976), no.~2, 248--260, 296.
  \MR{0412114}

\bibitem[Wan14]{wan}
Xin Wan, \emph{Iwasawa main conjecture for supersingular elliptic curves},
  2014, preprint, arXiv:1411.6352.

\end{thebibliography}
\end{document}